\documentclass[12pt]{amsart}
\usepackage[a4paper, margin=1in]{geometry}
\usepackage{graphicx}
\usepackage{amsmath, amsfonts, amssymb,amsthm}
\usepackage{mathrsfs}
\usepackage{tikz}
\usetikzlibrary{calc}
\usepackage[colorlinks=true, linkcolor=blue, citecolor=blue, urlcolor=blue]{hyperref}
\usepackage{enumitem}

\newcommand{\clos}{\operatorname{clos}}
\newcommand{\Int}{\operatorname{int}}
\newcommand{\R}{\mathbb{R}}
\newcommand{\disk}{\mathbb{D}}

\newtheorem{theo}{Theorem}[section]
\newtheorem{lemma}[theo]{Lemma}
\newtheorem{prop}[theo]{Proposition}

\newtheorem{MainTheorem}{Theorem}

\numberwithin{equation}{section}

\begin{document}
\title[Collinear Fractals and Bandt's Conjecture]{Collinear Fractals and Bandt's Conjecture}

\author[B. Espigul\'e]{Bernat Espigul\'e}
\address{Departament d'Inform\`atica, Matem\`atica Aplicada i Estad\'{\i}stica,
Universitat de Gi\-ro\-na, c/ Maria Aur\`elia Capmany 61, 17003 Girona, Spain.
ORCID 0000-0002-8036-7851}
\email{bernat@espigule.com}

\author[D. Juher]{David Juher}
\address{Departament d'Inform\`atica, Matem\`atica Aplicada i Estad\'{\i}stica,
Universitat de Gi\-ro\-na, c/ Maria Aur\`elia Capmany 61, 17003 Girona, Spain.
ORCID 0000-0001-5440-1705}
\email{david.juher@udg.edu}

\author[J. Salda\~na]{Joan Salda\~na}
\address{Departament d'Inform\`atica, Matem\`atica Aplicada i Estad\'{\i}stica,
Universitat de Gi\-ro\-na, c/ Maria Aur\`elia Capmany 61, 17003 Girona, Spain. ORCID 0000-0001-6174-8029}
\email{joan.saldana@udg.edu}

\thanks{This work has been funded by grants PID2023-146424NB-I00 of Ministerio de Ciencia,  Innovaci\'on y Universidades and 2021 SGR 00113 of Ge\-ne\-ra\-litat de Catalunya.}


\begin{abstract}
For a complex parameter $c$ outside the unit disk and an integer $n\ge2$, we examine the \emph{$n$-ary collinear fractal} $E(c,n)$, defined as the attractor of the iterated function system $\{\mbox{$f_k \colon \mathbb{C} \longrightarrow \mathbb{C}$}\}_{k=1}^n$, where $f_k(z):=1+n-2k+c^{-1}z$. We investigate some topological features of the connectedness locus $\mathcal{M}_n$, similar to the Mandelbrot set, defined as the set of those $c$ for which $E(c,n)$ is connected. In particular, we provide a detailed answer to an open question posed by Calegari, Koch, and Walker in 2017. We also extend and refine the technique of the ``covering property'' by Solomyak and Xu to any $n\ge2$. We use it to show that a nontrivial portion of $\mathcal{M}_n$ is regular-closed. When $n\ge21$, we enhance this result by showing that, in fact, the whole $\mathcal{M}_n\setminus\mathbb{R}$ lies within the closure of its interior, thus proving that the generalized Bandt's conjecture is true.
\end{abstract}


\maketitle

\thispagestyle{empty}

\section{Introduction}\label{sec1}
Polynomials with integer coefficients are fundamental objects in mathematics, serving as building blocks in various areas such as algebra, number theory, and geometry. The roots of these polynomials often exhibit fascinating and intricate patterns in the complex plane. In this paper, we uncover a deep connection between these roots and a class of self-similar sets that we call \textit{collinear fractals}. These fractals are generated by repeatedly applying simple mathematical transformations involving a complex parameter $c$ and an integer $n$. Exotic elements of the family include some self-affine tiles with a collinear digit set independently studied in~\cite{Akiyama2021}. Remarkably, the set $\mathcal{M}_n$ of parameters $c$ for which the corresponding fractal is connected can be identified with the set of roots of polynomials with integer coefficients restricted from $-n+1$ to $n-1$.  By exploring this bridge between algebra and geometry, we provide new insights into long-standing mathematical questions, demonstrating how the algebraic properties of polynomials shape the geometric structure of fractals and give rise to complex and beautiful sets.

The concept of visualizing roots of polynomials is not new, and numerous mathematical explorations have arisen from this idea, particularly in blogs and online mathematical discussions~\cite{Christensen2006,Baez2023}. Research on the so-called Littlewood polynomials, whose coefficients are $\pm 1$, produced some of the earliest $\mathcal{M}_n$-like imagery; see, for instance, the work of Peter and Jonathan Borwein~\cite{Bailey2007,Borwein2008}. Similarly, polynomials with coefficients restricted to $\{0, 1\}$, known as Newman polynomials, were thoroughly investigated in the seminal work of Odlyzko and Poonen~\cite{Odlyzko1993}. Other studies related to roots of polynomials include the \textit{Thurston’s Master Teapot}~\cite{Bray2021,Lindsey2022,LindseyTiozzoWu2024}, \textit{Algebraic Number Starscapes}~\cite{Harriss2022,Dorfsman22}, and the eigenvalues of \textit{Bohemian Matrices}~\cite{Sendra2021}.

For any integer $n \geq 2$, let $A_n$ be the set
\[
A_n := \{ n+1 - 2k \}_{k = 1}^n= \{ - n+1, - n+3, \dots, n - 3, n - 1 \}.
\]
Let $\disk := \{ z \in \mathbb{C} : |z| \leq 1 \}$ denote the closed unit disk. For any parameter $c\in\mathbb{C}\setminus\disk$, consider the \emph{iterated function system} (IFS) $\{\mbox{$f_t \colon \mathbb{C} \longrightarrow \mathbb{C}$}\}_{t\in A_n}$, where $f_t(z):=t+c^{-1}z$. The attractor $E(c,n)$ of this IFS is the unique nonempty compact set satisfying
\begin{equation}\label{eq:attractor}
E(c,n) = \bigcup_{t \in A_n} \left( t + \frac{E(c,n)}{c} \right).
\end{equation}

These sets $E(c,n)$, which we refer to as \textit{collinear fractals}, are fundamental examples of self-similar sets in the plane, and understanding their topological and fractal properties is of significant interest. For a geometric description of $E(c,n)$, label the $n$ first-level \emph{pieces} of $E(c,n)$ as
\begin{equation*}
   E_t(c,n)=f_{t}(E(c,n))=t+c^{-1}E(c,n), \mbox{ where } t\in A_n.
\end{equation*}
For $n$ odd, the central piece is $E_0(c,n)=c^{-1}E(c,n)$, a copy of $E(c,n)$ centered at $0$, scaled down by $|1/c|$ and rotated by $\arg(1/c)$. The neighboring pieces of $E_0(c,n)$ are $E_{-2}(c,n)$ and $E_2(c,n)$.
For $n$ even, $0\notin A_n$ and the central pieces are $E_{-1}(c,n)$ and $E_1(c,n)$, with centers at $-1$ and $1$ respectively. Each piece $E_t(c,n)$ is just a translated copy of $c^{-1}E(c,n)$, and with the exception of $E_0(c,n)$ when $n$ is odd, each piece $E_t(c,n)$ comes with an identical pair $E_{-t}(c,n)$ symmetrically centered on the opposite side of the real line. See Figures~\ref{fig:Ec4_example}~and~\ref{fig:Ec5_example}.

\begin{figure}
    \centering
\includegraphics[width=.8\linewidth]{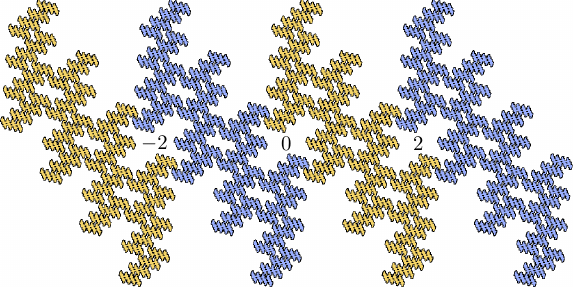}
    \caption{Example of a collinear fractal $E(c,4)$ for a specific parameter $c=(3 + i \sqrt{11})/2$, illustrating the intersection of neighboring pieces centered at $A_4=\{-3,-1,1,3\}$. By symmetry, the three main components of $\mathbb{C}\setminus E(c,4)$ are centered at $A_3=\{-2,0,2\}$.}
    \label{fig:Ec4_example}
\end{figure}

\begin{figure}
\centering
\includegraphics[width=\linewidth]{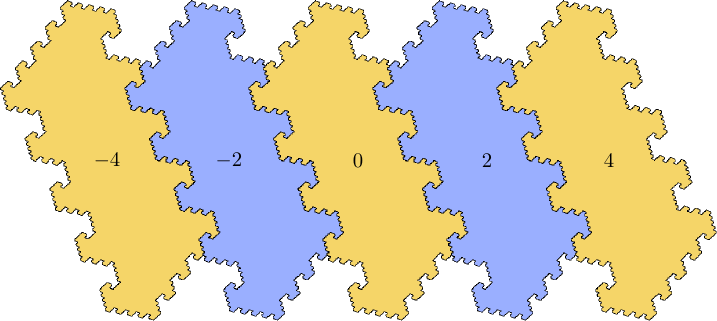}
\caption{Plane-filling collinear fractal $E(c,5)$ with $c=1 + 2i$. The first-level pieces are centered at $A_5= \{-4,-2,0,2,4\}$.
    }
    \label{fig:Ec5_example}
\end{figure}

A natural question is whether the attractor $E(c,n)$ is connected or completely disconnected. For IFS consisting of two maps ($n = 2$), there is a well-known dichotomy: the attractor is connected if and only if the images of the attractor under the two maps overlap, i.e., the two pieces $(-1 + c^{-1} E(c,2))$ and $(1 + c^{-1} E(c,2))$ intersect.

For a general IFS with $n>2$ contractions, this dichotomy does not hold in general: see, for instance, some examples of homogeneous self-similar sets in \cite{Solomyak1998MeasureFamilies}. However, for the collinear fractals $E(c,n)$ considered here, we have a similar criterion: we will show (Proposition~\ref{delicat}) that $E(c,n)$ is connected if and only if neighboring pieces on the right-hand side of (\ref{eq:attractor}) intersect. 

The \emph{connectedness locus} $\mathcal{M}_n$ is defined as the set of parameters $c$ for which $E(c,n)$ is connected:
\[
\mathcal{M}_n := \{ c \in \mathbb{C} \setminus \disk : E(c,n) \text{ is connected} \}.
\]
In the last decades, considerable efforts have been dedicated to understanding the topological properties of the set $\mathcal{M}_2$. It is worth noticing that, traditionally, the open unit disk has been the preferred parameter space for studying $\mathcal{M}_2$. However, using its geometric inversion $c\mapsto 1/c$ is advantageous to clarify the boundary of $\mathcal{M}_n$ and simplify the geometric arguments.

In 2008, Bandt and Hung \cite{Bandt2008FractalSets} introduced a family of self-similar sets known as \emph{$n$-gon fractals} for $n\in \mathbb{N}$ with $n\ge 2$ and parameterized by $\lambda\in\Int(\disk)\setminus\{0\}$, which reduce to our collinear fractals $E(1/\lambda,2)$ when $n=2$. The authors proved that the corresponding connectedness loci $\mathscr{M}_n$ are regular-closed for all $n\ne 4$:
\[
\clos(\Int(\mathscr{M}_n))=\mathscr{M}_n\ \mbox { for $n\ge3$ with $n\ne 4$.} 
\]

The regular-closedness of $\mathscr{M}_4$ was proved in 2020 by Himeki and Ishii \cite{Himeki2020M4Regular-closed} by extending some techniques from~\cite{Calegari2017RootsConjecture}. The extreme points of the $n$-gon fractals were characterized by Calegari and Walker in  \cite{Calegari2019ExtremeSets}. Bousch \cite{Bousch1992SurHolomorphe} proved that $\mathscr{M}_n$ is connected for any $n\ge 3$. The same author \cite{Bousch1993ConnexiteFonctions} also showed that $\mathcal{M}_2$ is connected and locally connected. Finally, Nakajima~\cite{Nakajima2024MandelbrotSeries} extended the work of Bousch proving the local connectedness of $\mathscr{M}_n$ for any $n\ge 3$. Therefore, the problems concerning the regular-closedness, connectedness, and local connectedness of $\mathscr{M}_n$ have been fully solved.

Going back to our sets $\mathcal{M}_n$, it turns out that Nakajima's general framework for studying the connectivity of the set of zeros of power series can be directly applied to our families of collinear fractals, which results in the sets $\mathcal{M}_n$ being connected and locally connected for any $n\ge 2$ (Theorem~\ref{thm:MLC} in Section~\ref{sec2}).

This paper focuses on understanding the connectedness loci $\mathcal{M}_n$ for any $n\geq2$.
The case $n=2$, sometimes referred to as the ``Mandelbrot set for a pair of linear maps'', has been extensively studied \cite{ Barnsley1985AMaps, Bousch1993ConnexiteFonctions,
Indlekofer1995OnSystems, Bandt2002OnMaps, Solomyak2003OnConvolutions, Solomyak2004MandelbrotGeometry,Solomyak2005OnSelf-similarity,Shmerkin2006ZerosSets,Bandt2008Self-similarOverlaps, Calegari2017RootsConjecture, Hare2015Two-dimensionalUniqueness, Hare2017OnExpansions, Shmerkin2016AbsoluteConvolutions, Silvestri2023AccessibilitySet}, 
but less is known for $n > 2$. Bandt \cite{Bandt2002OnMaps} conjectured that $\mathcal{M}_2 \setminus \mathbb{R}$ is contained in the closure of its interior, that is, the nonreal part of $\mathcal{M}_2$ is regular-closed. Solomyak and Xu \cite{Solomyak2003OnConvolutions} made significant progress towards this conjecture by showing that a nontrivial portion~$\mathcal{X}_2$ of $\mathcal{M}_2$ near the imaginary axis is the closure of its interior, 
\[
\mathcal{M}_2 \cap \mathcal{X}_2 \subset \clos\left( \Int(\mathcal{M}_2) \right).
\]

For a set $S\subset\mathbb{C}$, its \emph{set of differences} will be denoted as
\[ S\ominus S:=\{ p-q:\ p,q\in S\}. \]

In 2017, Calegari, Koch, and Walker \cite{Calegari2017RootsConjecture} proved Bandt's conjecture by introducing the technique of traps to certify interior points of $\mathcal{M}_2$. The authors also wondered about the properties of $E(c,2^k+1)$ for $k\ge0$ and its set of differences, which turns out to be $E(c,2^{k+1}+1)$, leaving further investigation of $\mathcal{M}_n$ for $n=2^k+1$ as an open problem.
In this paper, we address this not only for $n=2^k+1$ but also for any $n\geq 2$. We extend Solomyak and Xu's covering property lemma to all $n \geq 2$.
Specifically, we improve and adapt the techniques used in \cite{Solomyak2003OnConvolutions} to the general setting. As a result, we obtain that for $n \geq 21$, the covering property lemma suffices to prove the generalized Bandt's conjecture, namely, that $\mathcal{M}_n \setminus \mathbb{R}$ is contained in the closure of its interior.

\section{Preliminaries}\label{sec2}
In this section, we introduce some basic definitions and establish some properties of collinear fractals $E(c,n)$ and their connectedness locus $\mathcal{M}_n$.

\begin{prop}\label{delicat}
For any integer $n\ge2$, $E(c,n)$ is connected if and only if 
\[ \frac{E(c,n)}{c} \cap \left( 2 + \frac{E(c,n)}{c} \right) \neq \emptyset. \]
\end{prop}
\begin{proof}
It is immediate that the condition $(E(c,n)/c) \cap (2+E(c,n)/c)\neq\emptyset$ 
is equivalent to the property that, for any pair $\{t,t+2\}\subset A_n$, 
\begin{equation}\label{etiq} \left(t+\frac{E(c,n)}{c}\right) \cap \left( t+2 + \frac{E(c,n)}{c} \right) \neq \emptyset. \end{equation}
In other words, that any two neighboring pieces have a nonempty intersection.

Consider the \emph{connectivity graph} $G$, a combinatorial graph whose vertices are the $n$ elements of $A_n$ and there is an edge connecting a pair of vertices $\{i,j\}$ if and only if the corresponding first-level pieces $i+E(c,n)/c$ and $j+E(c,n)/c$ intersect. Now observe that either $G$ contains all the edges $\{i,i+2\}$ for $i=-n+1,-n+3,\ldots,n-3$ if (\ref{etiq}) holds, or $G$ is a collection of $n$ singletons (with no edges) otherwise. So, the proposition follows from the well-known fact that a self-similar set is connected if and only if the corresponding connectivity graph is connected \cite{Bandt1991Self-SimilarFractals,Hata1985OnSets}.
\end{proof}

An explicit representation of $E(c,n)$ is well known \cite{Solomyak1998MeasureFamilies} and is given by
\begin{equation}\label{eq:Ecn_explicit}
E(c,n) = \left\{ \sum_{k=0}^\infty a_k c^{-k} : a_k \in A_n \right\}.
\end{equation}

From (\ref{eq:Ecn_explicit}) and the fact that $A_n\ominus A_n=A_{2n-1}$ it easily follows that the set of differences of $E(c,n)$ is

\begin{equation}\label{dif}
E(c,n) \ominus E(c,n) = E(c, 2n - 1).
\end{equation}

Define now the set
\begin{equation}\label{eq:Dn}
D_n := \frac{1}{2} A_{2n - 1} =\left\{ 1-n, 2-n, \dots, -1, 0, 1, \dots, n-2, n-1 \right\}.
\end{equation}

\begin{figure}
    \centering
\includegraphics[width=\linewidth]{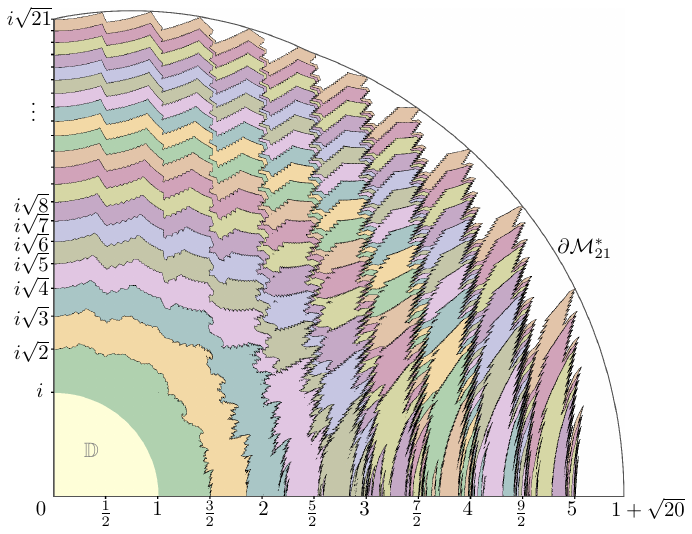}
    \caption{Superimposed arrangement of $\mathcal{M}_2$, $\mathcal{M}_3$, \dots, $\mathcal{M}_{21}$ constrained within the upper-right section of the complex plane. From Proposition~\ref{prop:Mn_nested} we know that the connectedness loci are nested. The illustration suggests the existence of infinitely many holes, and that, for $n\geq 4$, the intersection $\partial \mathcal{M}_n\cap\partial \mathcal{M}_{n+1}\setminus\mathbb{R}$ is nonempty.
    }
    \label{fig:Mn2to21}
\end{figure}

\begin{figure}
    \centering
\includegraphics[width=.9\linewidth]{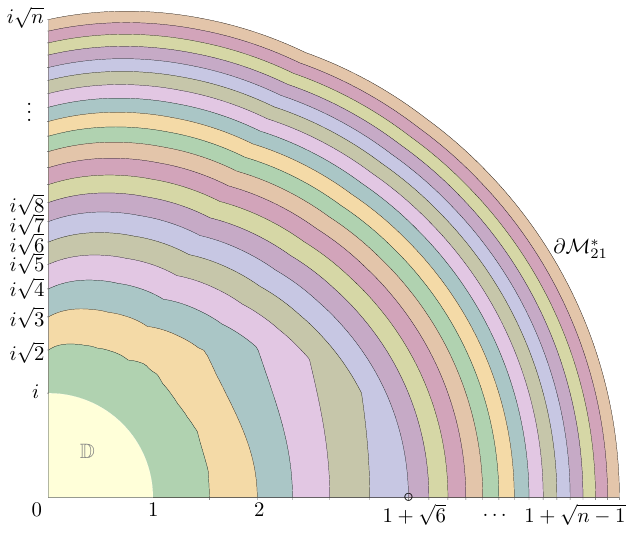}
    \caption{Convexity sets $\mathcal{M}_2^*$, $\mathcal{M}_3^*$, \dots, $\mathcal{M}_{21}^*$ constrained within the upper-right section of the complex plane. From Proposition~\ref{prop:upperbound} we know that $\mathcal{M}_n^*\setminus\mathbb{R}\subset B(1+\sqrt{n-1},0)$.
    }
    \label{fig:Mnstar2to21}
\end{figure}

Using these definitions, we can characterize the connectedness locus $\mathcal{M}_n$ as follows.

\begin{prop}\label{prop:Mn_characterization}
For any integer $n\ge2$,
\[ \mathcal{M}_n = \left\{ c \in \mathbb{C} \setminus \disk : \exists\ a_1, a_2,\ldots \in D_n\mbox{ such that } 1+ \sum_{k=1}^\infty a_kc^{-k} = 0 \right\}. \]
\end{prop}
\begin{proof}
From Proposition~\ref{delicat} and the characterization (\ref{eq:Ecn_explicit}) of $E(c,n)$ it follows that $E(c,n)$ is connected if and only if
\[
0 \in 2 + \left( \frac{E(c,n)}{c} \ominus \frac{E(c,n)}{c}\right) = 2 + \frac{E(c,2n - 1)}{c},
\]
where in the last equality we have used (\ref{dif}). This equality, again using (\ref{eq:Ecn_explicit}), implies that there exist $b_k \in A_{2n-1}$ and $a_k = b_k/2 \in D_n$ such that
\[
0 = 2 + \sum_{k=1}^\infty b_k c^{-k}=2\biggl(1 + \sum_{k=1}^\infty a_k c^{-k}\biggr).
\]
\end{proof}

The subsequent result is a direct consequence of Proposition~\ref{prop:Mn_characterization} and the straightforward observation that $D_n\subset D_{n+1}$. See Figure~\ref{fig:Mn2to21}.

\begin{prop}\label{prop:Mn_nested}
$\mathcal{M}_n\subset \mathcal{M}_{n+1}$ for any integer $n\ge2$.
\end{prop}

In order to estimate a bounding region for $\mathcal{M}_n$, we will use a deep connection between $\mathcal{M}_n$ and the set of zeros of a larger class of power series investigated earlier by Beaucoup et al.~\cite{Beaucoup1998PowerRay}. Let $I_n:=[1-n,n-1]\subset\mathbb{R}$ be the convex hull of $D_n$. We define the \emph{convexity set} as
\begin{equation}\label{eq:Mnstar}
\mathcal{M}_n^* := \left\{ c \in \mathbb{C} \setminus \disk \colon \exists\   a_1, a_2, \ldots \in I_n \mbox{ such that } \ 1 + \sum_{k=1}^\infty a_k c^{-k} = 0 \right\}.
\end{equation}
In Figure~\ref{fig:Mnstar2to21} we have represented the convexity sets $\mathcal{M}_i^*$ for $2\le n\le 21$.

In what follows, the closed disk of radius $r$ centered at $z_0$ will be denoted by $B(r,z_0)$. 

\begin{prop}\label{prop:upperbound}  
For any integer $n\ge2$,
\[ \mathcal{M}_n\subset\mathcal{M}_n^*\subset B(1 + \sqrt{n - 1},0)\cup [-n,n]. \]
\end{prop}
\begin{proof}
The inclusion $\mathcal{M}_n\subset\mathcal{M}_n^*$ comes directly from the definition (\ref{eq:Mnstar}), the fact that $D_n\subset I_n$ and Proposition~\ref{prop:Mn_characterization}. 

Let us now see that $\mathcal{M}_n^*\cap\mathbb{R}\subset[-n,n]$. In \cite{Beaucoup1998PowerRay}, the authors study the set $Z_g$ of all zeros with an absolute value smaller than 1 of power series of the form
\[ 1+\sum_{k=0}^\infty a_kz^k, \mbox{ where }a_k\in[-g,g]\]
for a positive real number $g$. Note that, by definition (\ref{eq:Mnstar}), $\mathcal{M}_n^*$ is nothing more than the inversion of the set $Z_g$ when $g=n-1$. Let $r_g(\phi)$ be the infimum of the moduli of all numbers in $Z_g$ with a fixed argument $\phi$. Theorem~A of \cite{Beaucoup1998PowerRay} states that $r_g(\phi)\ge 1/(g+1)$, with the equality achieved when $\phi=0$. By inversion and taking $g=n-1$, this proves that $\mathcal{M}_n^*\cap\mathbb{R}\subset[-n,n]$. 

On the other hand, Theorem~B of \cite{Beaucoup1998PowerRay} states that $r_g(\phi)>1/(\sqrt{g}+1)$ for all $\phi\in(0,\pi)$. Taking inverses and setting $g=n-1$, this implies that $\mathcal{M}_n^*\setminus\mathbb{R}\subset B(1+\sqrt{n-1},0)$. 
\end{proof}

Additional bounds for $\mathcal{M}_n$ are provided by the following result, which corresponds to Lemma~2.5 of \cite{Solomyak1998MeasureFamilies}.

\begin{prop}\label{prop:trivial} For any integer $n\ge2$, 
\begin{enumerate}
\item[(i)] $\{ c\in\mathbb{C} : 1<|c|<\sqrt{n}\}\subset\mathcal{M}_n$
\item[(ii)] $(-n,-1)\cup(1,n)\subset\mathcal{M}_n\cap\mathbb{R}$.
\end{enumerate}
\end{prop}

By Propositions~\ref{prop:upperbound} and~\ref{prop:trivial}(i), the boundary of $\mathcal{M}_n$ lies \emph{essentially} in the closed annulus $\left\{ c \in \mathbb{C} : \sqrt{n}\leq|c| \leq 1 + \sqrt{n - 1} \right\}$. In fact, the contention is not strict due to the \emph{antennae} given by Proposition~\ref{prop:trivial}(ii). This peculiar feature of the sets $\mathcal{M}_n$ restricted to the real axis has been previously described for $n=2$ in \cite{Barnsley1985AMaps,Bandt2002OnMaps, Shmerkin2006ZerosSets, Calegari2017RootsConjecture}. See Figure~\ref{fig:M8} for a representation of the connectedness locus $\mathcal{M}_8$ with the interval $(1+\sqrt{7},8]$ removed from the positive real antenna. 

Finally, as mentioned in Section~\ref{sec1}, Nakajima's study of the set of zeros of power series \cite{Nakajima2024MandelbrotSeries} can be applied to our families of collinear fractals. Specifically, in view of Proposition~\ref{prop:trivial}(ii) and taking into account that $1\in D_n$, we can use Theorem~B of \cite{Nakajima2024MandelbrotSeries} with $L=1/\sqrt{n}$ to obtain the following result. 

\begin{theo}\label{thm:MLC}
$\mathcal{M}_n$ is connected and locally connected for any integer $n\ge2$.
\end{theo}

\begin{figure}
    \centering
\includegraphics[width=.9\linewidth]{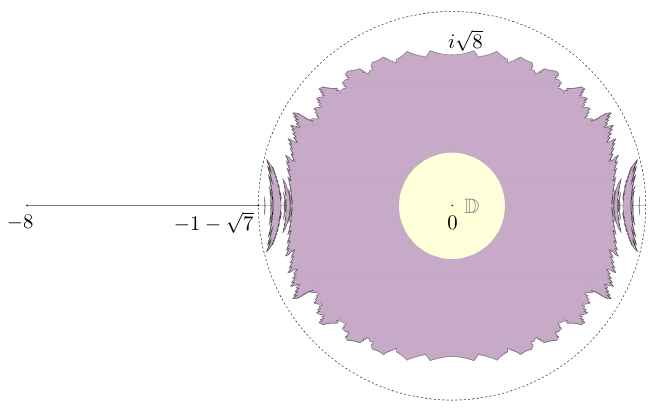}
    \caption{The connectedness locus $\mathcal{M}_8$. From Proposition~\ref{prop:upperbound} we know that $\mathcal{M}_8\setminus\mathbb{R}$ is contained in a disk of radius $1+\sqrt{7}$.}
    \label{fig:M8}
\end{figure}

\section{Statement of the main result}\label{sec3}
To state the main result of this paper, we need to define a particular region $\mathcal{X}_n\subset\mathbb{C}$ as follows. 
Inspired by the methods in \cite{Solomyak2003OnConvolutions}, we define the set $\mathcal{X}_n$ as
\begin{equation}\label{eq:Omega_n}
\begin{aligned}
\mathcal{X}_n :&=\biggl\{ c=re^{i\theta}\in\mathbb{C}\setminus\R:1<r\leq \sqrt{2 n-\sin ^2 \theta }-| \cos \theta |\biggr\}\\
&= \left\{ c\in \mathbb{C} \setminus \bigl(\R\cup\disk\bigr) : |c + 1| \leq \sqrt{2n},\ |c - 1| \leq \sqrt{2n} \right\} \\
&=\biggl(B\bigl(\sqrt{2n},-1\bigr)\cap B\bigl(\sqrt{2n},1\bigr)\biggr)\setminus \bigl(\R\cup\disk\bigr).
\end{aligned}
\end{equation}
See Figure~\ref{fig:Omega_n} for an illustration of the region $\mathcal{X}_8$.

\begin{MainTheorem}\label{thm:main}
$\mathcal{M}_n \cap \mathcal{X}_n \subset \clos\left( \Int(\mathcal{M}_n) \right)$ for any integer $n\ge2$. Moreover, if $n\ge21$, then $\mathcal{M}_n \cap \mathcal{X}_n = \mathcal{M}_n \setminus \mathbb{R}$.
\end{MainTheorem}

An immediate consequence of Theorem~\ref{thm:main} is that the generalized Bandt's conjecture is true for $n\ge21$.

The proof of Theorem~\ref{thm:main} provides numerous specific examples of interior points in $\mathcal{M}_n$. Let $\widehat{\mathcal{M}}_n$ be the set of zeros of {\em polynomials} with coefficients
in $D_n$ defined as
\begin{equation}\label{eq:Mnhat}
\widehat{\mathcal{M}}_n:= \Bigg\{c \in \mathbb{C}\setminus\disk:\ \exists\ a_1,\ldots,a_m\in D_n \mbox{ such that }
c^m\biggl(1 + \sum_{k=1}^m a_k c^{-k}\biggr) = 0 \Bigg\}\,.
\end{equation}
We will show that every point in $\widehat{\mathcal{M}}_n \cap \mathcal{X}_n$ is located within the interior of $\mathcal{M}_n$.

\begin{figure}
\centering
\includegraphics[width=.7\linewidth]{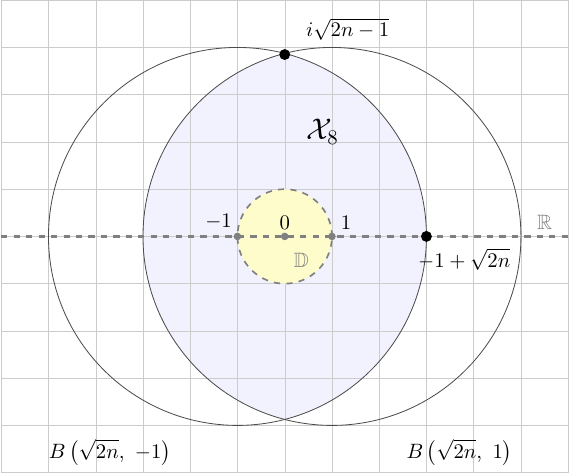}
\caption{Illustration of the region $\mathcal{X}_n$ defined in~(\ref{eq:Omega_n}) for $n=8$. The shaded area represents $\mathcal{X}_8$, which is the intersection of two disks of radius $\sqrt{2n}$ centered at $-1$ and $1$, excluding the unit disk $\disk$ and the real axis $\mathbb{R}$.}
\label{fig:Omega_n}
\end{figure}

\section{Proof of Theorem~\ref{thm:main}}\label{sec4}
To prove Theorem~\ref{thm:main} we need some preliminary lemmas, adapting techniques from \cite{Solomyak2003OnConvolutions}. The first result is standard, see for instance Lemma~7 of \cite{Indlekofer1995OnSystems}.

\begin{lemma}\label{lem:attractor_inclusion}
If $F \subset \mathbb{C}$ is a compact set, $c \in \mathbb{C} \setminus \disk$, and
\[
F \subset \bigcup_{t \in A_n} \left( t + \frac{F}{c} \right),
\]
then $F \subset E(c,n)$.
\end{lemma}

\begin{lemma}\label{lem:Mn_inclusion}
If $F \subset E(c, 2n - 1)$ and there exist $a_1,a_2,\ldots,a_m\in D_n$ such that
\[
1 + \sum_{k=1}^m a_k c^{-k} \in \frac{1}{2 c^{m+1}} F,
\]
then $c \in \mathcal{M}_n$.
\end{lemma}

\begin{proof}
Our hypothesis, via (\ref{eq:Ecn_explicit}) and (\ref{eq:Dn}), imply that
\begin{equation}\label{qtal} 1 + \sum_{k=1}^m a_k c^{-k} = \frac{1}{c^{m+1}}\sum_{k=0}^\infty b_k c^{-k} = \sum_{k=0}^\infty b_k c^{-k-(m+1)} \end{equation}
for some sequence $\{b_k\}_{k=1}^\infty$ such that $b_k \in D_n$.
Now take $q_k:=a_k\in D_n$ for $1\le k\le m$ and $q_k:=-b_{k-(m+1)}\in -D_n=D_n$ for all $k>m$. Rewrite (\ref{qtal}) as
\[ 1 + \sum_{k=1}^\infty q_k c^{-k} = 0.\]
Since $q_k\in D_n$ for all $k\geq1$, 
Proposition~\ref{prop:Mn_characterization} tells us that $c\in\mathcal{M}_n$.
\end{proof}

\begin{figure}
    \centering
\includegraphics[width=.5\linewidth]{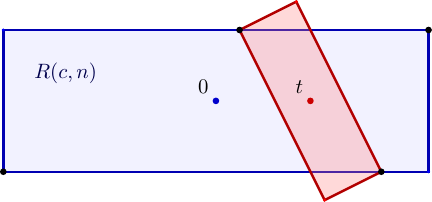}
    \caption{Geometric configuration of the rectangle $R(c,n)$ and its image $t+c^{-1}R(c,n)$ for $t \in A_n$.}
    \label{fig:R5t}
\end{figure}

\begin{figure}
    \centering
\includegraphics[width=.7\linewidth]{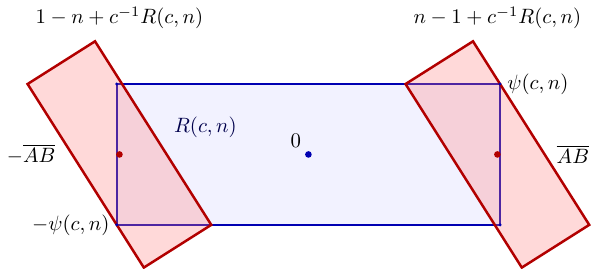}
    \caption{An illustration of the critical case $n= |c|^2 + 2|\operatorname{Re}(c)|$ in the proof of Lemma~\ref{lem-comput}.}
    \label{fig:psiedges}
\end{figure}

Next we state a \emph{covering property} which will be the key tool to prove Theorem~\ref{thm:main}. Its proof is much simpler than the one given by Solomyak and Xu for the case $n=3$, see~\cite[Lemma 3.3]{Solomyak2003OnConvolutions}. For instance, they lacked an explicit parameterization for the covering rectangle. Let $R(c,n)\subset \mathbb{C}$ denote the rectangle centered at the origin with vertices 
\begin{equation}\label{eq:rect}
\biggl\{\psi(c,n),\overline{\psi(c,n)},-\psi(c,n),-\overline{\psi(c,n)}
\biggr\},
\end{equation}
where $\overline{\psi(c,n)}$ is the complex conjugate of the vertex $\psi(c,n)$ defined as
\begin{equation}\label{eq:psi}
 \psi(c,n) := \begin{cases} c(n+1)/(1+c) & \text{if } \operatorname{Re}(c) \geq 0 \\ c(n+1)/(1-c)  & \text{if } \operatorname{Re}(c) < 0, \end{cases}
\end{equation}
where $\operatorname{Re}(c)$ denotes the real part of $c$. The expression of the vertex $\psi(c,n)$ in terms of $c$ and $n$ was obtained by imposing the geometric conditions prescribed in the proof of the covering property (Lemma~\ref{lem-comput}).

\begin{lemma}[covering property]\label{lem-comput} 
For any $c\in \mathcal{X}_{\frac{n+1}{2}}$, the rectangle $R(c,n)$ is covered by its $n$ first-level images 
\[
R(c,n) \subset \bigcup_{t\in A_n}\biggl(t+\frac{R(c,n)}{c}\biggr).
\]
\end{lemma}
\begin{proof}
Recall that $R(c,n)$ is the rectangle (\ref{eq:rect}) centered at the origin with vertices defined in (\ref{eq:psi}). One can easily check that, for any $t\in A_n$, there is a pair of diagonally opposing vertices of the rectangle $t+c^{-1}R(c,n)$ placed along the horizontal lines containing the upper and lower edges of $R(c,n)$. The remaining pair of vertices of $t+c^{-1}R(c,n)$ are placed above and below those lines (see Figure~\ref{fig:R5t}).

To show that the $n$ rectangles $t+c^{-1}R(c,n)$ cover $R(c,n)$ as long as the parameter $c$ is in $\mathcal{X}_{\frac{n+1}{2}}$, consider the \emph{outer boundary} of $\mathcal{X}_{\frac{n+1}{2}}$ defined as $\partial\mathcal{X}_{\frac{n+1}{2}}\setminus(\partial\disk\cup\mathbb{R})$ which, after some algebraic manipulations, can be implicitly parameterized by
\begin{equation}\label{c}
n= |c|^2 + 2 |c| |\cos\theta| = |c|^2 + 2 |\operatorname{Re}(c)|,
\end{equation}
where $\theta$ denotes the argument of $c$.

Now observe that if the parameter $c$ satisfies (\ref{c}), then 
\[ |c|= \sqrt{n+1-\sin ^2 \theta }-| \cos \theta | \] 
and the vertices $\psi(c,n)$ and $-\psi(c,n)$ intersect the edges $-\overline{AB}$ and $\overline{AB}$ of the leftmost and rightmost rectangles, where $A:=n-1+c^{-1}\psi(c,n)$ and $B:=n-1-c^{-1}\overline{\psi(c,n)}$. See Figure~\ref{fig:psiedges}. Moreover, the $n$ rectangles $t+c^{-1}R(c,n)$ intersect tangentially side by side, thus \textit{critically} covering the rectangle $R(c,n)$, see Figure~\ref{fig:R5covering}. 

If $c\in\Int(\mathcal{X}_{\frac{n+1}{2}})$, then \[1<|c|< \sqrt{n+1-\sin ^2 \theta }-| \cos \theta |\]
and $R(c,n)$ is covered by its $n$ images (Figure~\ref{fig:R5covering}, bottom). Finally, if $c\notin\mathcal{X}_{\frac{n+1}{2}}\cup(\mathbb{D}\cup\mathbb{R})$ then 
\[|c|>\sqrt{n+1-\sin ^2 \theta }-| \cos \theta |\]
and the $n$ rectangles $t+c^{-1}R(c,n)$ are disjoint (Figure~\ref{fig:R5covering}, top), thus not covering $R(c,n)$.
\end{proof}

\begin{figure}
\centering
\includegraphics[width=\linewidth]{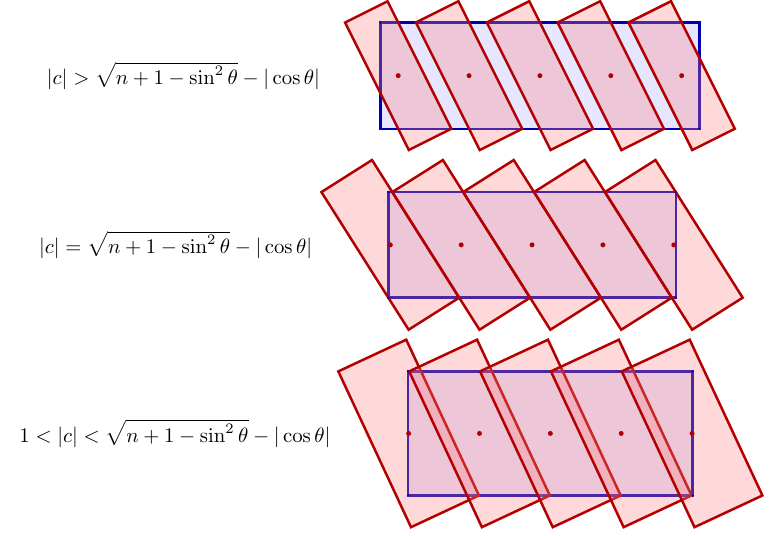}   
    \caption{An illustration of the proof of the covering property (Lemma~\ref{lem-comput}).
    }
    \label{fig:R5covering}
\end{figure}

The following lemma is a standard consequence of Rouch\'e's Theorem. Recall that $\widehat{\mathcal{M}}_n$ denotes the set of zeros of {\em polynomials} with coefficients
in $D_n$ given in~(\ref{eq:Mnhat}). 

\begin{lemma} \label{lem-rouche}
$\mathcal{M}_n = \clos(\widehat{\mathcal{M}}_n)$.
\end{lemma}

Now we have all the necessary ingredients to prove Theorem~\ref{thm:main}.

\begin{proof}[{\bf Proof of Theorem~\ref{thm:main}}]
Let us prove that $\mathcal{M}_n \cap \mathcal{X}_n \subset \clos(\Int(\mathcal{M}_n))$ for any $n\ge2$. In view of Lemma~\ref{lem-rouche}, it is enough to show that $\Int(\mathcal{X}_n)\cap \widehat{\mathcal{M}}_n\subset\Int(\mathcal{M}_n)$. So, let $c_0\in\Int(\mathcal{X}_n)\cap \widehat{\mathcal{M}}_n$. We must see that there is an open neighborhood $U$ of $c_0$ contained in $\mathcal{M}_n$.

Since $c_0\in \widehat{\mathcal{M}}_n$, there exist $a_1, a_2,\ldots,a_m\in D_n$ such that
$$
p(c_0)=0, \mbox{ where } p(z) =  1 + \sum_{k=1}^m a_k z^{-k}.
$$
Note that $\frac{1}{2c_0^{m+1}}R(c_0,2n-1)$ is a solid (with nonempty interior) rectangle centered at $0$. Hence, 
\begin{equation}\label{hola2}
p(c_0)=0\in\Int\biggl(\frac{1}{2c_0^{m+1}}R(c_0,2n-1)\biggr).
\end{equation}
We are assuming that $c_0\in\Int(\mathcal{X}_n)$. Moreover, the functions (\ref{eq:psi}) that define the rectangles $R(c,2n-1)$ are continuous with respect to $c$. These facts, together with (\ref{hola2}) and the continuity of $p(z)$, imply that there is an open neighbourhood $U$ of $c_0$ such that $U\subset\mathcal{X}_n$ and 
\begin{equation}\label{hola3}
p(c)\in\frac{1}{2c^{m+1}}R(c,2n-1)\mbox{ for every }c\in U. 
\end{equation}
On the other hand, from Lemmas~\ref{lem:attractor_inclusion} and~\ref{lem-comput} it follows that
\begin{equation}\label{hola4}
R(c,2n-1)\subset E(c,2n-1)\mbox{ for all }c\in\mathcal{X}_n.
\end{equation}
Finally, (\ref{hola3}) and (\ref{hola4}) allow us to use Lemma~\ref{lem:Mn_inclusion}, which tells us that $c\in \mathcal{M}_n$ for every $c \in U$.

Let us now prove the second statement of the theorem. We must show that for $n\geq 21$ we have $\mathcal{M}_n \cap \mathcal{X}_n=\mathcal{M}_n \setminus\R$. The inclusion $\mathcal{M}_n \cap \mathcal{X}_n\subset\mathcal{M}_n \setminus\R$ is obvious because $\mathcal{X}_n\cap\mathbb{R}=\emptyset$. Hence, we only need to show that $\mathcal{M}_n \setminus\R\subset \mathcal{X}_n$ for $n\geq 21$. 
It is easy to check that $1+\sqrt{n-1}<-1 + \sqrt{2 n}$ for $n\ge21$. See Figure~\ref{fig:R21} for an example. In consequence, 
\begin{equation}\label{estimate2}
    B\bigl( 1+\sqrt{n-1},0\bigr)\subset B\bigl(\sqrt{2n},-1\bigr)\cap B\bigl(\sqrt{2n},1\bigr)\mbox{ for }n\ge21.
\end{equation}
Since $\mathcal{X}_n=\biggl(B\bigl(\sqrt{2n},-1\bigr)\cap B\bigl(\sqrt{2n},1\bigr)\biggr)\setminus\bigl(\R\cup\disk\bigr)$ and, from Proposition~\ref{prop:upperbound}, we know that $\mathcal{M}_n \setminus \mathbb{R}$ is contained within the disk $B\bigl( 1+\sqrt{n-1},0\bigr)$, it follows that $\mathcal{M}_n \setminus\R\subset \mathcal{X}_n$.
\end{proof}

\begin{figure}
\centering    
\includegraphics[width=\linewidth]{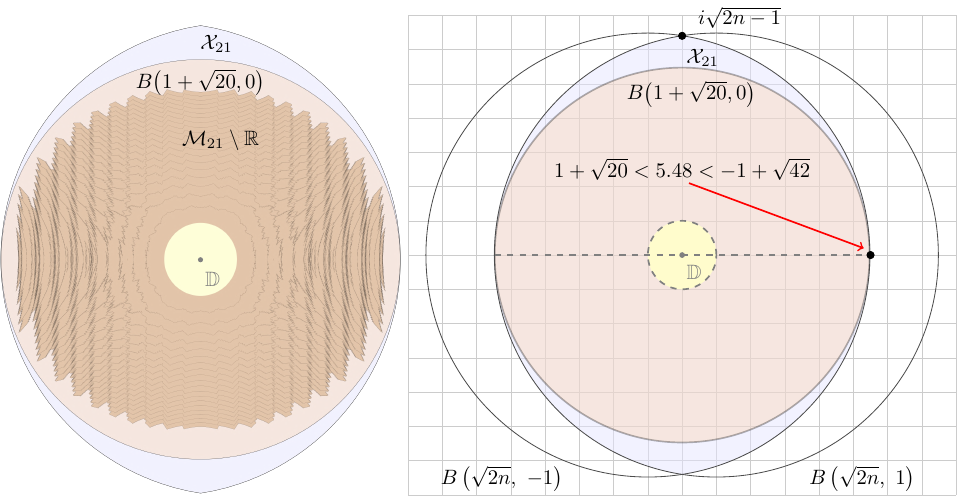}
    \caption{The set $\mathcal{M}_{21}\setminus\R$ contained in $\mathcal{X}_{21}$. Since $1+\sqrt{20}<-1 + \sqrt{42}$, the region $\mathcal{X}_{21}$ contains $B(1+\sqrt{20},0)\setminus\bigl(\R\cup\disk\bigr)$ which in turn contains $\mathcal{M}_{21}\setminus\R$ by Proposition~\ref{prop:upperbound}.
    }
    \label{fig:R21}
\end{figure}

\section{Remarks and suggestions on further research}
The connection we established between the convexity set $\mathcal{M}_n^*$ and the connectedness locus $\mathcal{M}_n$ is clearer when we consider the following characterizations of $\mathcal{M}_n$ and $\mathcal{M}_n^*$ involving the set $E(c,2n-1)$ of differences of $E(c,n)$ and its convex hull $H(c,2n-1)$.
\begin{equation*}
\begin{aligned}
    \mathcal{M}_n&=\bigl\{c\in\mathbb{C}\setminus\mathbb{D}:2c\in E(c,2n-1)\bigr\},\\
    \mathcal{M}_n^*&=\bigl\{c\in\mathbb{C}\setminus\mathbb{D}:2c\in H(c,2n-1)\bigr\}.
\end{aligned}
\end{equation*}
In particular, the condition $2c\in E(c,2n-1)$ implies that there is an asymptotic self-similarity between $\mathcal{M}_n$ and $E(c,2n-1)$. For each $c\in  \partial\mathcal{M}_n\setminus\mathbb{R}$ we have $2c\in \partial E(c,2n-1)$, and a neighborhood of $c$
from $\mathcal{M}_n$ looks asymptotically similar to a neighborhood of $2c$ from $E(c,2n-1)$, observe the animation~\cite{Espigule2024AsymptoticYoutu.be/11NZDHNahJs}.

For $3 \leq n < 21$, we conjecture that the set $\mathcal{M}_n \setminus (\mathbb{R} \cup \mathcal{X}_n)$ is regular-closed. One could try to extend the partial results of Nguyen Viet Hung for $n=2$, who, as part of his PhD thesis~\cite{Hung2007PolygonFractals.}, obtained three new regions in addition to $\mathcal{X}_2$. Note that for any $n$, it holds that $ \mathcal{X}_n \subset B(\sqrt{2n - 1}, 0) $, and when $ |c| < \sqrt{2n - 1} $, it is expected that $ E(c, 2n - 1) $ possesses a nonempty interior, given that the similarity dimension of the IFS exceeds 2. Computational evidence~\cite[personal notes]{Espigule2024MandelbrotEcn} suggests that by replacing the rectangle \( R(c,n) \) with the parallelogram \( P(c,n) \) centered at the origin with vertices
\begin{equation*}
\biggl\{\, n - 1 + c^{-1},\; n - 1 - c^{-1},\; -n + 1 - c^{-1},\; -n + 1 + c^{-1} \,\biggr\},
\end{equation*}
we have that \( P(c, 2n - 1) \subset E(c, 2n - 1) \) for any non-real \( c \in \mathcal{M}_n \). These findings are part of ongoing work and will be further investigated in future research.

Additionally, inspired by Solomyak and Xu's investigation into complex Bernoulli convolutions~\cite{Solomyak2003OnConvolutions}, exploring the measures supported on $E(c,n)$ and their absolute continuity could prove to be a productive avenue for future research.

\section{Conclusions}

In this paper we have introduced the family of \emph{collinear fractals} $E(c,n)$ defined as the compact sets invariant under the iterated function system $\{f_t(z):=t+c^{-1}z\}_t$, where $c$ is a complex parameter outside the unit disk and $t$ ranges over the symmetric set of integers $\{-n+1,-n+3,\ldots,n-3,n-1\}$. 

For any integer $n\ge2$, we define the connectedness locus $\mathcal{M}_n$ as the set of parameters $c$ for which $E(c,n)$ is connected. Among other results, we have proven that a nontrivial portion of $\mathcal{M}_n$ is contained in the closure of its interior for any $n\ge2$. In addition, we prove that, when $n\ge 21$, \emph{the whole} $\mathcal{M}_n\setminus\mathbb{R}$ lies in fact within the closure of its interior. In other words, that the generalized Bandt's conjecture about the regular-closedness of $\mathcal{M}_n$ is true.

\bibliographystyle{amsalpha}
\bibliography{references}

\end{document}